\documentclass[12pt]{amsart}
\usepackage{amsmath}
\usepackage{amsmath, pb-diagram}
\usepackage[all,cmtip]{xy}
\usepackage{amssymb}
\usepackage{amscd}

\newtheorem{theorem}{Theorem}[section]
\newtheorem{lemma}{Lemma}[section]
\newtheorem{proposition}{Proposition}[section]

\newtheorem{corollary}{Corollary}[section]
\newcommand{\Hom}{\operatorname{Hom}}
\newcommand{\End}{\operatorname{End}}

\newcommand{\Soc}{\operatorname{Soc}}

\newcommand{\Ker}{\operatorname{Ker}}

\begin{document}

\author{Abyzov A. N.}
\address{Department of Algebra and Mathematical Logic, Kazan (Volga Region) Federal University, 18 Kremlyovskaya str., Kazan, 420008 Russia}
\email{aabyzov@ksu.ru}
\title{Almost projective and almost injective modules}
\maketitle
\begin{abstract} 
We describe  rings over which every right module  is almost injective. We give
a description of rings over which every simple  module is a  almost projective.
\end{abstract}
\vskip 2cm

Let $M,N$ be right $R$-modules. A module $M$ is called {\it almost $N$- injective}, if for any submodule $N'$ of $N$ and any homomorphism $f : N' \to M$, either  there exists a homomorphism  $g : N \to M$ such that  $f=g\iota$ or  there exists a nonzero idempotent $\pi\in \End_R(N)$ and a homomorphism $h : M\to  \pi(N)$ such that $hf=\pi\iota,$ where  $\iota:N'\to N$ is the natural embedding. 
A module $M$ is called {\it almost injective} if it is  almost $N$-injective for every right $R$-module $N$. Dually, we define the concept of almost projective modules. A module $M$ is called {\it almost $N$-projective,} if for any natural homomorphism $g : N \to N/K$ and any
homomorphism $f : M\to N/K$,  either  there exists a homomorphism  $h : M \to N$ such that  $f=gh$ or  there exists a non-zero direct summand $N'$ of $N$  and a homomorphism $h' : N'\to  M$ such that $g\iota=fh',$ where $\iota:N'\to N$ is the natural embedding.
A module $M$ is called {\it almost projective} if it is almost 
$N$-projective for every right $R$-module $N$.

The concepts of almost injective module and almost projective module were studied in the works \cite{1}-\cite{9} by Harada and his colleagues. Note that, in \cite{9} an almost projective right $R$-module is defined as a module which is  almost $N$-projective to every finitely generated right $R$-module $N$. In recent years, almost injective modules were considered in \cite{10}-\cite{14}.  The problem of the description of the rings over which all modules are almost injective was studied in \cite{12}. In some special cases, this problem was solved in \cite{12}. In particular, in the case of semiperfect rings. In this article, we study the structure of the rings over which every module is almost injective, in general.  We also give the characterization of the module $M$ such that every simple module is almost projective (respectively, almost injective) in the category $\sigma(M)$.

Let $M,N$ be right $R$-modules. We denote by $\sigma(M)$ the full subcategory of $Mod$-$R$ whose objects are
all $R$-modules subgenerated by $M$. If $N\in \sigma(M)$ then the injective hull of the module $N$ in $\sigma(M)$ will be denoted by $E_M(N)$. The Jacobson radical of the module $M$ is denoted by $J(M)$.

{\it The Loewy series} of a module $M$ is the ascending chain of submodules
\begin{center}
$0=\Soc _0(M) \subset \Soc _1(M)=\Soc (M) \subset \ldots \subset \Soc _{\alpha}
(M) \subset \Soc _{\alpha+1}(M) \subset \ldots$,
\end{center}
\noindent
where $\Soc _{\alpha+1 }(M)/\Soc _{\alpha }(M)=\Soc (M/\Soc _{\alpha}(M))$
for all ordinal numbers $\alpha$ and $\Soc _\alpha (M)=
\bigcup\limits_{\beta < \alpha} \Soc _\beta (M)$ for  a limit ordinal number $\alpha$. Denote by $L(M)$ the submodule $\Soc_{\xi }(M)$, where $\xi$ is the smallest ordinal such that $\Soc _{\xi} (M)=\Soc _{\xi+1}(M)$. The module $M$ is semiartinian if and only if $M= L(M)$.
In this case $\xi$ is called {\it the Loewy length} of module $M$ and is denoted by
$\mathrm{Loewy} (M)$.
The ring $R$ is called {\it right semiartinian} if the module $R_{R}$ is semiartinian.

The present paper uses standard concepts and notations of ring theory (see, for example \cite{15}-\cite{17} ).

\section{ Almost projective modules}

A module $M$ is called an {\it $I_{0}$-module} if every its nonsmall submodule contains nonzero direct summand of the module $M$.

\begin{theorem} For a module $M$, the following assertions are equivalent:

\begin{itemize}

     \item [1)] Every simple module in the category $\sigma(M)$ is almost projective.

    \item [2)] Every module in the category $\sigma(M)$ is either a semisimple module or contains a nonzero  $M$-injective submodule.
    
    \item [3)] Every module in the category $\sigma(M)$ is an $I_0$-module.
    
    \end{itemize}
    \end{theorem} 
        
\begin{proof}  1)$\Rightarrow $2)  Let $xR\in \sigma(M)$ be a non-semisimple cyclic module. Then the module $xR$ contains an essential maximal submodule $N.$  Let $f:E_M(xR)$ $\rightarrow$ $E_M(xR)/N$ be the natural homomorphism and $\iota :xR/N\to E_M(xR)/N$ be the embedding. Assume that there exists a homomorphism $g:xR/N\to E_M(xR)$ such that $fg=\iota.$ Since $g(xR/N)\subset f^{-1}(xR/N)=xR$ and $N$ is an essential submodule of $xR,$ then $g(xR/N)\subset N.$
Consequently $fg=0,$ which is impossible. Since the module $xR/N$ is almost projective, for some nonzero  direct summand $N'$ of $E_M(xR)$ and homomorphism $h:N'\to xR/N$ we get $\iota h=f\iota',$ where $\iota' :N'\to E_M(xR)$ is the embedding. Consequently $f(N')\subset xR/N,$ i.e. $N'\subset f^{-1}(xR/N)=xR.$

2)$\Rightarrow $3) The implication follows from \cite[Theorem 3.4]{18}.

3)$\Rightarrow $1) Let $S$ be a simple right $R$-module,  $f:A$ $\rightarrow$ $B$ be an epimorphism right $R$-modules and $g :S\to B$ be a homomorphism. Without loss of generality, assume that $g\neq 0.$ If $\Ker(f)$ is not an essential submodule of $f^{-1}(g(S))$, then  there exists a simple submodule $S'$ of $f^{-1}(g(S))$ such that $f(S')=g(S).$ In this case, obviously, there is a homomorphism $h:S\to A$ such that $fh=g.$ Assume $\Ker(f)$ is an essential submodule of $f^{-1}(g(S)).$ Then $f^{-1}(g(S))$ is a non-semisimple module and by \cite[Theorem 3.4]{18}, $f^{-1}(g(S))$ contains a nonzero  injective submodule $A'.$ There exists a homomorphism $g':g(S)\to S$ such that $gg'(s)=s$ for all $s\in g(S).$ Then $g(g'f_{\mid A'})=f\iota,$ where $\iota :A'\to A$ is the embedding and $f_{\mid A'}:A'\to g(S)$ is the restriction of the homomorphism $f$ to $A'$.
\end{proof} 
 
\begin{corollary} Every right $R$-module is an $I_{0}$-module if and only if every simple right $R$-module is almost projective.
\end{corollary}

A right $R$-module $M$ is called a {\it $V$-module
(or {\it cosemisimple})} if every proper submodule of $M$ is an intersection of maximal submodules of $M.$ A ring 
$R$ is called a {\it right $V$-ring} if $R_R$ is a $V$-module.
It is known that a right $R$-module $M$ is a $V$-module if and only if every simple right $R$-module is
$M$-injective. A ring $R$ is called a {\it right $SV$-ring} if  $R$ is a right semiartinian right $V$-ring.

\begin{theorem}  For a regular ring $R$, the following assertions are equivalent:
\begin{itemize}

      \item [1)] Every right $R$-module is an $I_{0}$-module.
      
      \item [2)] $R$ is a right $SV$-ring.
    
      \item [3)] Every right $R$-module is almost projective.
    
      \item [4)] Every simple right $R$-module is almost projective.
    
    \end{itemize} 
  \end{theorem} 
  
\begin{proof}  The equivalence 1)$\Leftrightarrow$2) follows from \cite[theorem 3.7]{18}. The implication 3)$\Rightarrow $4) is obvious. The implication 4)$\Rightarrow $1) follows from Theorem 1.1. 
   
2)$\Rightarrow$3) Let $S$ be a simple right $R$-module. We claim that the module $S$ is almost projective. Let $f:A$ $\rightarrow$ $B$ be an epimorphism right $R$-modules and $g :S\to B$ be a homomorphism.  Without loss of generality, assume that $\Ker(f)\neq 0.$   Then $\Ker(f)$ contains a simple injective submodule $S'$ and for the homomorphism $h=0\in \Hom(S', S)$ we get $f\iota=gh,$ where $\iota:S\to A$ is the natural embedding.
 \end{proof}     
    
A ring $R$ is called a {\it $I$-finite} (or {\it orthogonally finite}) if it does not contain an infinite set of orthogonal nonzero idempotents.

\begin{theorem}  For a $I$-finite ring $R$, the following assertions are equivalent:
\begin{itemize}

     \item [1)] Every right $R$-module is almost projective.
     
     \item [2)] Every simple right $R$-module is almost projective.

    \item [3)] $R$ is an artinian serial ring and $J^2(R)=0$.
    
    \end{itemize}
 \end{theorem} 
    
\begin{proof}  The implicatio 1)$\Rightarrow $2) is obvious. 

 2)$\Rightarrow $3) By Theorem 1.1 and \cite[13.58]{16}, $R$ is a semiperfect ring. Then by \cite[Theorem 3.2]{18}, $R$ is an artinian serial ring and $J^2(R)=0$.

3)$\Rightarrow $1) Let $M$ be a right $R$-module. We claim that the module $M$ is almost projective. Let $f:A$ $\rightarrow$ $B$ be an epimorphism of right $R$-modules and $g :M\to B$ be a homomorphismм. If $f^{-1}(g(M))$ is a semisimple module, then it is obvious that there is a homomorphism $h$ such that $g=fh.$ Assume $f^{-1}(g(M))$ is a non-semisimple module. Then the module $f^{-1}(g(M))$ contains an injective and projective local submodule $L$ of length two. Since $L$ is a projective module, then there is a homomorphism $h':L\to g^{-1}(f(L))$ such that $f\iota=g_{\mid g^{-1}(f(L))}h',$ where $\iota:L\to A$ is the natural embedding.
\end{proof} 

\section{Almost $V$-modules}

A right $R$-module $M$ is called an {\it almost $V$-module} if every simple right $R$-module is almost $N$-injective for every module  $N\in \sigma(M).$  A ring $R$ is called a  {\it right almost $V$-ring} if every simple right $R$-module is almost injective. Right almost $V$-rings have been studied in \cite{13}.

\begin{lemma}  For a module $M$, the following assertions are equivalent:

  \begin{itemize}
    \item [1)] $M$ is not a $V$-module.
    \item [2)] There exists a submodule $N$ of the module $M$ such that the factor module $M/N$ is an uniform, $Soc(M/N)$ is a simple module and $M/N\neq Soc(M/N)$.
     \end{itemize}
\end{lemma}  

\begin{proof}  The implicatio 2)$\Rightarrow $1) is obvious.

1)$\Rightarrow $2) Since $M$ is not a $V$-module, there is a submodule $M_0$ such that $J(M/M_0)\neq 0.$ Without loss of generality, assume that $J(M/M_0)$ contains a simple submodule $S.$ Let $S'$ be a complement of submodule $S$ in $M/M_0.$ Then $(M/M_0)/S'$ is an uniform module, $Soc((M/M_0)/S')$ is a simple module and $(M/M_0)/S'\neq Soc((M/M_0)/S')$.
\end{proof} 
 
\begin{proposition}  Let $M$ be an almost $V$-module. Then:
\begin{itemize}

      \item [1)] The Jacobson radical $J(N)$ of every module $N\in \sigma(M)$ is semisimple.
      
      \item [2)] The factor module $N/J(N)$ of every module $N\in \sigma(M)$ is a $V$-module.
    
      \item [3)] The injective hull $E_M(S)$ of every simple module $S\in \sigma(M)$ is either a simple module or a local $M$-projective module of length two.
      
    \end{itemize}
    
\end{proposition}    
    
\begin{proof} 1)  Assume that in the category $\sigma(M)$ there exists a module whose Jacobson radical is not semisimple. Then there exists  a module  $N\in \sigma(M)$ and a non-zero element $x\in J(N)$ such that the module $xR$ contains an essential maximal submodule $A.$ Let $B$ be a complement of submodule $A$ in $N.$ Consider the homomorphism $f:xR\oplus B\to xR/A$ is defined by
$f(xr+b)=xr+A,$ where $r\in R, b\in B.$ Assume that there exists a homomorphism $g: N\to xR/A$ such that $g\iota=f,$ where $\iota:xR\oplus B\to E$ is the natural embedding. Since $x\in J(N),$  $g\iota(x)=0.$ On the other hand, $f(x)\neq 0.$ This is a contradiction. If there is a nonzero idempotent $\pi\in \End_R(N)$ and a homomorphism $h:xR/A\to \pi(N)$ such that $\pi\iota=hf,$ then $hf(\pi(N)\cap (A\oplus B))=0$ and $\pi\iota(\pi(N)\cap (A\oplus B))\neq 0$ for a nonzero  submodule $\pi(N)\cap (A\oplus B),$  that is impossible. Thus a Jacobson radical $J(N)$ of every module $N\in \sigma(M)$  is semisimple.

2)  Let $N\in \sigma(M)$ be a module and $S\in \sigma(M)$ be a simple module, $N_0$ be a submodule of $N'=N/J(N)$ and $f:N_0\to S$ be a homomorphism. We show that  there exists a homomorphism $g$ such that $f=g\iota,$ where $\iota:N_0\to N'$ is the natural embedding. Without loss of generality, assume that $N_0$ is essential in $N'$ and $f\neq 0.$ Assume $\Ker(f)$ is an essential submodule of $N_0.$ If there exists a non-zero idempotent $\pi\in \End_R(N')$ and a homomorphism $h:S\to \pi(N')$ such that $\pi\iota=hf,$ then $hf(\pi(N')\cap \Ker(f))=0$ and $\pi\iota(\pi(N')\cap \Ker(f))\neq 0$  for a nonzero submodule  $\pi(N')\cap \Ker(f),$  that is impossible. Thus there exists a homomorphism $g$ such that $f=g\iota.$ Assume $\Ker(f)$ is not an essential submodule of  $N_0.$ Then there exists a simple module $S'$ such that $N_0=\Ker(f)\oplus S'.$ Assume that there exists a non-zero idempotent $\pi\in \End_R(N')$ and a homomorphism $h:S\to \pi(N')$ such that $\pi\iota=hf.$  Since $\Ker(f)\oplus S'$ is essential in $N',$ $\Ker(f)\subset (1-\pi)N'$ and $(1-\pi)N'\oplus \pi(S')=(1-\pi)N'\oplus S'$, then $\pi(S')$ is essential in $\pi(N').$ Since $J(N')=0,$ we get $\pi(S')=\pi(N')$, and consequently $N'=(1-\pi)N'\oplus S'.$ Then there exists a  $g:(1-\pi)N'\oplus S'\to S$ homomorphism  is defined by
$g(n+s)=f(s),$ where $n\in (1-\pi)N', s\in S'$ such that $f=g\iota.$ Hence $N'$ is a $V$-module.

3)  Let $S\in \sigma(M)$ be a simple module and $E_M(S)\neq S.$ By 2),  $J(E_M(S))=S.$ Let $A_1, A_2$ be maximal submodules of $E_M(S).$ From the proof of \cite[13.1(a)]{19}, we see that $\End_R(A_1), \End_R(A_2)$ are local rings.  Assume that $A=A_i\oplus A_j$ is a $CS$-module, where $i,j\in \{1,2\}$. Let $B$ is a closed submodule of $A$ and $A\neq B.$ Then $B$ is complement of some simple submodule $S'$ in $A.$ Consider the homomorphism $f:S'\oplus B\to S'$ is defined by the formula  
$f(s+b)=s,$ where $s\in S', b\in B.$ Since  $S'\in J(A)$ and $S'$ is an almost $A$-injective module, there is a non-zero idempotent $\pi\in \End_R(A)$ and a homomorphism $g\in \Hom_R(S, \pi(A))$ such that $gf=\pi\iota,$ where $\iota:S'\oplus B\to A$ is the natural embedding. It's clear that $B\subset (1-\pi)A$ and $S\cap (1-\pi)A=0.$ Consequently $B=(1-\pi)A.$ Thus $A$ is a $CS$-module. From \cite[7.3(ii)]{19} and the fact that every monomorphism $\phi:A_i\to A_j$ is an isomorphism we deduce that $A_i$ is an $A_j$-injective module. If $A_1\neq A_2$ then by \cite[16.2]{17}, $A_1$ is an $A_1+A_2$-injective, which is impossible. Thus the module $E_M(S)$ has an unique maximal submodule, and consequently $E_M(S)$ is a local module of length two. We claim that $E_M(S)$ is  projective  in the category  $\sigma(M).$ Let $N$ be a submodule of $E_M(S)\oplus M$ such that $N+M=E_M(S)\oplus M$ and $\pi:E_M(S)\oplus M\to E_M(S)$ be the natural projection. Assume that $J(N) \subset N\cap M.$ Since $N/J(N)$ is a $V$-module, $\pi(N)=E_M(S)$ is a $V$-module, which is impossible. Thus there exists a simple submodule $S'$ of $J(N)$ such that $S'\cap M=0.$ Let $A$ be a complement of submodule $S'$ in $N$ such that $M\cap N\subset A.$ 
Consider the homomorphism $f:S'\oplus A\to S'$ is defined by 
$f(s+a)=s,$ where $s\in S', a\in A.$ 
Since $S'\subset J(N)$ and $S'$ is an almost $N$-injective module, there is a non-zero idempotent $\pi'\in \End_R(N)$ and a homomorphism $g\in \Hom_R(S', \pi'(N))$ such that $gf=\pi'\iota,$ where $\iota:S'\oplus A\to N$  is the natural embedding. Since $A\subset (1-\pi')(N), S'\subset J(N)$ and $A\oplus S'$ is an essential submodule of $N,$ we deduce that $\pi'(S)$ is essential in $\pi'(N)$ and $\pi(S')\neq\pi'(N).$ Since
$$
\pi'(N)\cap A=\pi'(N)\cap N\cap M=\pi'(N)\cap M=0
$$ 
and  $\mathrm{lg}(E_M(S))=\mathrm{lg}(\pi'(N))=2,$ we
have  $\pi(\pi'(N))=E_M(S)$. Then  $\pi'(N)\oplus M=E_M(S)\oplus M.$ By \cite[41.14]{17}, the module $E_M(S)$ is  projective  in the category 
$\sigma(M).$
\end{proof}

\begin{theorem}  For a module $M$, the following assertions are equivalent:
\begin{itemize}

     \item [1)] $M$ is an almost $V$-module.

     \item [2)] Every module in the category $\sigma(M)$ is either a $V$-module or contains a nonzero  direct summand which is a projective object in the category $\sigma(M)$.
    
    \item [3)] There exist an independent set of local submodules $\{A_i\}_{i\in I}$ of the module $M$ such that:
    \begin{itemize}
    \item [a)] $A_i$ is both an $M$-injective and an $M$-projective module of length two for all $i\in I$;
    \item [b)] $J(M)=\oplus_{i\in I}J(A_i);$
    \item [c)] $M/J(M)$ is a $V$-module.
     \end{itemize}
   
    \end{itemize}
\end{theorem} 
        
\begin{proof}   1)$\Rightarrow $2)  Let $N$ be a module in the category $\sigma(M)$ which is not a $V$-module. Then by Lemma 2.1 and Proposition 2.1,  there is a submodule $N'$ of $N$ such that the factor module $N/N'$ is  nonzero and projective  in the category $\sigma(M).$ Consequently the natural epimorphism $f:N\to N/N'$ splits and the module $N$ contains a nonzero  direct summand which is a projective  in the category $\sigma(M)$.

2)$\Rightarrow $1) Let $M$ be a right $R$-module and $S$ be a simple right $R$-module. We claim that $S$ is an almost $M$-injective module. Let $M_0$ be a submodule of $M$ and $f:M_0\to S$ be a homomorphism. Without loss of generality, assume that $f\neq 0,$ $M_0$ is an essential submodule of $M$ and $E_M(S)\neq S.$ There is a homomorphism $g:M\to E_M(S)$ such that $g\iota=\iota'f,$ where $\iota:M_0\to M$ and $ \iota':S\to E_M(S)$  the  natural embeddings. Assume that $S\neq g(M).$  Then by the condition 2), $g(M)$ is a projective module. Consequently $M=\Ker(g)\oplus M'.$ Since $M_0$ is an essential submodule of $M,$ then $M_0\cap M'$ is a simple module and $f_{\mid M_0\cap M'}:M_0\cap M'\to S$ is an isomorphism.  Then $M_0=(M_0\cap M')\oplus \Ker(f).$ Let $\pi:\Ker(g)\oplus M'\to M'$ be the natural projection. Then $\pi\iota=f_{\mid M_0\cap M'}^{-1}f.$

1)$\Rightarrow $3)  By Zorn's Lemma, there is a maximal independent set of submodules $\{A_i\}_{i\in I}$ of the module $M$ such that $A_i$ is a local module of length two for all $i\in I.$  According to Proposition 2.1, $M/J(M)$ is a $V$-module and $A_i$ is both an $M$-injective and an $M$-projective module for all $i\in I$. Assume that 
$J(M)\neq \oplus_{i\in I}J(A_i).$ Then by the condition 1), there is a simple submodule $S$ of $M$  such that $S\subset J(M)$ and $S\cap \oplus_{i\in I}J(A_i)=0.$ Let $S'$ be a complement of submodule $S$ in $M$ such that it contains $\oplus_{i\in I}J(A_i).$ Then $M/S'$ is not a simple module, which is an essential extension of the simple module $(S+S')/S'.$ By Proposition 2.1,  $M/S'$ is an $M$-projective module of length two. Consequently, there is a local submodule of length two $L$ of $M$ such that $M=L
\oplus S'.$ This contradicts with the choice of the set  $\{A_i\}_{i\in I}.$
Thus $J(M)=\oplus_{i\in I}J(A_i).$

3)$\Rightarrow $2) Let $S\in \sigma(M)$ be a simple module and $E_M(S)\neq S.$ By \cite[16.3]{17}, there exists an epimorphism $f:\oplus_{i\in I'}M_i\to E_M(S),$ where $M_i=M$ for all $i\in I'.$ Since $E_M(S)$ is not a $V$-module, by \cite[23.4]{17},   $f\epsilon_i(J(M))\neq 0$ for some $i\in I',$ where $\epsilon_i:M_i\to \oplus_{i\in I'}M_i$ is a natural embedding. Then, by the conditions a) and b) of 3), $E_M(S)\cong A_i$ for some $i\in I.$ Thus every essential extension of a simple module in the category $\sigma(M)$ is either a simple or a local $M$-projective module of length two. Then the implication follows directly from Lemma 2.1.
\end{proof}

\begin{corollary}  For a ring $R$, the following assertions are equivalent:
\begin{itemize}

    \item [1)]  $R$ is a right almost $V$-ring.
   \item [2)] Each right $R$-module is either a $V$-module or contains a nonzero  direct summand which is a projective module.
    \item [3)] There exist a set of orthogonal idempotents $\{e_i\}_{i\in I}$ of the ring $R$ such that:
    \begin{itemize}
    \item [a)] $e_i R$ is a local injective right $R$-module of length two for every $i\in I$;
    \item [b)] $J(P)=\oplus_{i\in I}J(e_i R);$
   \item [c)] $R/J(R)$ is a right $V$-ring.
     \end{itemize}
    
    \end{itemize}
\end{corollary} 

\begin{theorem}  For a right noetherian ring $R$, the following assertions are equivalent:
\begin{itemize}

     \item [1)] Every right $R$-module is a direct sum of an injective module and a $V$-module.

    \item [2)] Every right $R$-moduleis a direct sum of a projective module and a $V$-module.

    \item [3)] $R$ is a right almost $V$-ring.

\end{itemize}

\end{theorem} 

\begin{proof}   3)$\Rightarrow $1), 2) By Zorn's Lemma, there is a maximal independent set of local submodules of length two $\{L_i\}_{i\in I}$ of the module $M$. Since $R$ is a right noetherian ring, by \cite[6.5.1]{15},  there exists a submodule $N$ of $M$ such that $M=\oplus_{i\in I}L_i\oplus N.$ By Proposition 2.1 3),  $\oplus_{i\in I}L_i$ is both injective and projective. We claim that $N$ is a $V$-module. Assume that $N$ is not a $V$-module. Then by the Proposition 2.1 3) and Lemma 2.1, there exists a factor module $N/N_0$ of $N$  which is a local projective module of length two. Consequently, the module $N/N_0$ is isomorphic to a submodule of $N,$ which contradicts the choice of the set $\{L_i\}_{i\in I}.$ Thus $N$ is a  $V$-module.

2)$\Rightarrow $3) Let $S$ be a right simple module. Assume that $E(S)\neq S.$ By the condition 2), $E(S)$ is a projective module and by \cite[7.2.8]{15}, $\End_R(E(S))$ is a local ring. Then by \cite[11.4.1]{15},  $E(S)$ is a local module. If $J(E(S))$ is not a simple module, then by the condition 2), the module $E(S)/S$ is projective, and consequently $S$ is a direct summand of $E(S),$ which is  impossible. Thus the injective hull of a every simple right $R$-module is either a simple or a projective module of length two. Consequently 
$R$ is a right almost $V$-ring by \cite[Theorem 3.1]{13}.

1)$\Rightarrow $3) Since $R_R$ is a noetherian module then by the condition 1),  $R_R=M \oplus N,$ where $M$ is a finite direct sum of uniform injective modules and $N$ is a $V$-module. By \cite[7.2.8, 11.4.1]{15},  $M=L_1\oplus\ldots\oplus L_n,$ where $L_i$ is a local module for every $1\leq i\leq n.$  Assume that $J(L_{i_0})$ is nonzero  and is not a simple module for some $i_0.$ Then there is a non-zero element $r\in J(L_{i_0})$ such that $rR\neq J(L_{i_0}).$ Let $T$ be maximal submodule of $rR.$ By the condition 1),  the injective hull of every simple right $R$-module is either a simple module or a module of length two. Then the local module $L_{i_0}/T$ is not an injective module and it is not a $V$-module, which contradicts to condition 1). From these considerations, it follows that there exists a family of orthogonal idempotents $e_1,\ldots e_n$ of ring $R$ satisfying the condition a) and b) of Corollary 2.2, and $R_R/J(R)$ is the direct sum of a semisimple module and a $V$-module. By \cite[23.4]{17},  $R/J(R)$ is a right $V$-ring. Then, by Corollary 2, $R$ is an almost right $V$-ring.
 \end{proof}

\begin{theorem}  For a regular ring $R$, the following assertions are equivalent:
\begin{itemize}

     \item [1)] $R$ is a right $V$-ring.

    \item [2)] Every right $R$-module is a direct sum of an injective module and a $V$-module.

     \item [3)] Every right $R$-module is a direct sum of a projective module and a $V$-module.
    
    \item [4)]   $R$ is a right almost  $V$-ring.
    
    \end{itemize}
\end{theorem}

\begin{proof}  The implications 1)$\Rightarrow $2), 1)$\Rightarrow $3), 1)$\Rightarrow $4) are obvious.

2)$\Rightarrow $1) Assume that the ring $R$ is not a right $V$-ring. Then $E(S)\neq S$ for some simple right $R$-module $S.$ By the condition 2) we have $\oplus_{i=1}^{\infty}L_i=M\oplus N,$ where $L_i\cong E(S)$ for every $i,$ $M$ is an injective module and $N$ is a $V$-module. Since $J(\oplus_{i=1}^{\infty}L_i)$ is essential in $\oplus_{i=1}^{\infty}L_i,$ then $J(\oplus_{i=1}^{\infty}L_i)\cap N=J(N)$ is essential in $N$,  and consequently  $N=0.$ Let $I=\{r\in R\mid\ E(S)r=0\}.$ We can conside the module $\oplus_{i=1}^{\infty}L_i$ as a right module over the ring $R/I.$ Assume that $R/I$ is not a semisimple artinian ring. Then the ring $R/I$ contains a
countable set of non-zero orthogonal idempotents $\{e_i\}_{i=1}^{\infty}.$ For every $i\in \mathbb{N}$, there is an element $l_i\in L_i$ such that $l_i e_i\neq 0. $ Since the right $R/I$-module $\oplus_{i=1}^{\infty}L_i$ is injective, there exists a homomorphism $f:R/I_{R/I}\to \oplus_{i=1}^{\infty}L_i,$ such that $f(e_i)=l_i e_i$ for all $i.$ Since $f(R/I_{R/I})\subset \oplus_{i=1}^{n}L_i$ for some $n\in\mathbb{N},$ we obtain a contradiction with the fact that  $l_i e_i\neq 0 $ for all $i\in \mathbb{N}.$ Thus $R/I$ is a semisimple artinian ring. Consequently $E(S)=S.$ This contradiction shows that $R$ is a right $V$-ring.

3)$\Rightarrow $1)  Assume that the ring $R$ is not a right $V$-ring. Then by Lemma 2.1, there exists a right ideal $I$ of $R$ such that the right $R$-module $R/I$ is an uniform, is not a simple module and $\Soc(R/I_R)$ is a simple module. Then, by the condition 3), the module $R/I$ is projective, and consequently $R/I_R$ is isomorphic to a submodule of $R_R,$ which is impossible. This contradiction shows that $R$ is a right $V$-ring.

The implication 4)$\Rightarrow $1) follows directly from Corollary 2.1.

\end{proof} 

\section{ Rings Over Which Every Module Is Almost Injective}

Let $M$ be a right $R$-module. Denote by $SI(M)$ the sum of all simple injective submodules of the module $M.$ Clearly, $SI(R_R)$ is ideal of ring $R.$

\begin{lemma} Let $R$ be a ring with the following properties:

\begin{itemize}
\item [a)] in the ring $R$ there exists a finite set of orthogonal idempotents $\{e_i\}_{i\in I}$ such that $e_i R$ is local injective right $R$-module of length two, for each $i\in I$ and $J(R)=\oplus_{i\in I}J(e_iR)$;
\item [b)] $R/J(R)$ is a right $SV$-ring and $Loewy(R_R)\leq 2$;
\item [c)] $R/SI(R_R)$ is a right artinian ring.
\end{itemize}
Then we have the following statement:

\begin{itemize}
\item [1)] the injective hull of every simple right $R$-module  is either a simple module or a local projective module of length two;

\item [2)] every right $R$-module is a direct sum of a injective module and a $V$-module;

\item [3)]  every right $R$-module is a direct sum of a projective module and a $V$-module;

\item [4)] if $S$  a simple submodule of the right $R$-module $N,$ $S\subset J(N)$ and $S\cap N'=0$ for some submodule $N'$ of $N,$ then there are submodules $L, N''$ of  $N$ such that   $L$ is a local module of length two, $S\subset L, N'\subset N''$ and $N=N''\oplus L.$
\end{itemize}
\end{lemma}

\begin{proof} 1) Let $S$ be a simple right $R$-module and $E(S)\neq S.$ Since $R/J(R)$ is a right $V$-ring and $J(R)$ is a semisimple right $R$-module, then $E(S)S'\neq 0$ for some simple submodule of $S'$ of right $R$-module $J(R)_R$. From condition a) it follows that $S'$ is essential in some injective local submodule of the module $\oplus_{i\in I}e_iR.$ Therefore, $E(S)\cong e_{i_0}R$ for some $i_0\in I.$ Thus injective hull of every simple right $R$- module is either a simple module or a local projective module of length two.

2), 3) Let $M$ be a right $R$-module. By Lemma of Zorn there is a maximal independent set of submodules of $\{L_i\}_{i\in I}$ of a module $M$ such that $L_i$ is a local injective module of length two, for each $i\in I.$ Clearly, $E(\oplus_{i\in I}L_i)SI(R)=0.$ Then from the condition c) it follows that $E(\oplus_{i\in I}L_i)=\oplus_{i\in I}L_i.$ Therefore $M=\oplus_{i\in I}L_i\oplus N$ for some submodule $N$ of a module $M$. It is clear that module $\oplus_{i\in I}L_i$ is  injective and projective. If $N$ is $V$-module, then from Lemma 2.1 and condition 1) follows that for some submodule $N_0$ of the module $N$ factor module $N/N_0$ is a local projective module of length two. Therefore $N=N_0\oplus L$ where $L$ is a injective local module of length two, which impossible.  Thus $N$ is a $V$-module.

4) From conditions 1) and 2), it follows that $S\subset L$ where $L$ is a local injective submodule of a module $N$ of length two. Let $L'$ is a complement of $L$ in  $N$ which contains the submodule $N'.$ Then $(S+L')/L'$ is a essential  submodule of $N/L'$   and $N/L'\neq (S+L')/L'.$ From condition 1), it follows that $N/L'$ is a local module of length two. Therefore, the natural homomorphism $f:N\to N/L'$ induces an isomorphism $f_{\mid L}:L\to N/L'.$ Then $N=L\oplus L'.$~\hfill$\square$
\end{proof}

\begin{lemma} Let $M$ be a right $R$-module and $N$  be a injective submodule of $M$. If $N'$ is submodule of  $M$ and $N'\cap N=0,$ then $N'\subset N''$ and $M=N''\oplus N$  for some submodule $N''$ of  $M$ 
\end{lemma}

\begin{proof} Let $M'$ is a complement of $N$ in $M$ which contains the submodule $N'.$ Then $E(M)=E(N')\oplus N$ and $M=(E(N')\cap M) \oplus N.$~\hfill$\square$
\end{proof}

\begin{theorem} For a ring $R$ the following conditions are equivalent:
\begin{itemize}

\item [1)] Every right  $R$-module is almost injective.

\item [2)] $R$ is a right semiartinian  ring, $Loewy(R_{R})\leq 2$ and every right $R$-module is a direct sum of an injective module and a $V$-module.

\item [3)] $R$ is a right semiartinian  ring, $Loewy(R_{R})\leq 2$ and every right $R$-module is a direct sum of a projective module and a $V$-module.

\item [4)] The ring $ R $ satisfies the following conditions:
\begin{itemize}
\item [a)] in the ring $R$ there exists a finite set of orthogonal idempotents $\{e_i\}_{i\in I}$ such that $e_i R$ is a local injective right $R$-module of length two, for each $i\in I$ and $J(R)=\oplus_{i\in I}J(e_iR)$;
\item [b)] $R/J(R)$ is a right $SV$-ring and $Loewy(R_R)\leq 2$;
\item [c)] $R/SI(R_R)$ is a  right artinian ring.
\end{itemize}

\item [5)] The ring $R$ is isomorphic to the ring of formal 
matrix  \\ $\left(\begin{tabular}{c c} $T$ & $_TM_S$\\
0& $S$\end{tabular}\right)$, where
\begin{itemize}
\item [a)] $S$ is a right $SV$-ring and $Loewy(S)\leq 2$;
\item [b)] for some ideal $I$ of a ring $S$ the equality $MI=0$ holds and the ring $\left(\begin{tabular}{c c} $T$ & $_TM_{S/I}$\\
0& $S/I$\end{tabular}\right)$ is an artinian serial,
with the square of the Jacobson radical equal to zero.
\end{itemize}
\end{itemize}
\end{theorem}

\begin{proof} the Implication 4)$\Rightarrow $2) and 4)$\Rightarrow $3) follow from Lemma 2.

1)$\Rightarrow $4) From corollary 2.1 it follows that $R/J(R)$ is a right $V$-ring. According to \cite[proposition 2.6]{12} $Loewy(R_R)\leq 2.$
Then $R_R/Soc(R_R)$ is a semisimple module of finite length, and from corollary 2.1  follows that the ring $R$ contains a finite set of orthogonal idempotents $\{e_i\}_{i\in I}$ satisfying  the condition a) of 4). Therefore, $R_R=\oplus_{i\in I}e_i R\oplus A$, where $A$ is a  semiartinian  right $R$-module and $Loewy(A)\leq 2$. As $AJ(R)=0,$ then, by corollary 2.1, $A$ is a $V$-module. Suppose that $Soc(A)$ contains an infinite family of primitive orthogonal idempotents $\{f_i\}_{i\in I'}$ such that $f_iR\neq E(f_iR)$ for each $i\in I'.$ Let $B$ is a complement of $\oplus_{i\in I'}f_iR$ in $R_R,$ which contains the $J(R).$ Consider the homomorphism $f:\oplus_{i\in I'}f_iR\oplus B\to \oplus_{i\in I'}E(f_iR)$, defined by
$f(r+b)=r,$ where $r\in \oplus_{i\in I'}f_i R, b\in B.$ Assume that $\iota:\oplus_{i\in I'}f_iR\oplus B\to R_R$ is a natural embedding.  If there exists a homomorphism $g:R_R\to \oplus_{i\in I'}E(f_iR)$ such that $f=g\iota$ then $f(\oplus_{i\in I'}f_i R)\subset g(R_R)\subset \oplus_{i\in I''}E(f_i R),$   where  $I''\subset I'$.Therefore $\mid I''\mid<\infty,$ which is impossible. Since the module $\oplus_{i\in I'}E(f_iR)$ is a almost $R_R$-injective, then there exists non-zero idempotent $\pi\in End_R(R_R)$ and a homomorphism $h:\oplus_{i\in I'}E(f_iR)\to \pi(R_R)$ such that $\pi\iota=hf.$ Since $\oplus_{i\in I'}f_iR\oplus B$ is essential in $R_R,$ then $\pi\iota\neq 0.$ Therefore, $h\neq 0.$ Then $h(E(f_{i_0}R))\neq 0$ for some $i_0\in I'.$ Since $\pi\iota(J(R))=hf(J(R))=0,$ then $J(\pi(R_R))=0.$ From proposition 2.1 it follows that $E(f_{i_0}R)$ is a local projective module of length two.
Since $J(\pi(R_R))=0,$ then $Ker(h_{\mid E(f_{i_0}, R)})$ and $Im(h_{\mid E(f_{i_0}, R)})$ is a simple modules. Then
$Im(h_{\mid E(f_{i_0}, R)})$ is a direct summand of the module $R_R.$ Therefore, $Ker(h_{\mid E(f_{i_0}R)})$ is a direct summand of the module $E(f_{i_0} R),$ which is impossible. Thus, $Soc(A)=SI(R_R)\oplus B$ where $B$ is a module of finite length. Since $A/Soc(A), Soc(A)/SI(R_R)$ is a modules of finite length, then $A/SI(R_R)$ is a module of finite length. Therefore, $R/SI(R_R)$ is right artinian ring.

4)$\Rightarrow $1)  Suppose that the ring $R$ satisfy the condition 4) and $M, N$ are right $R$-modules. We claim that $M$ is an almost $N$-injective module. Let $N_0$ is a submodule of $N,$ and $\iota:N_0\to N$ be the natural embedding and $f:N_0\to M$ is a homomorphism. Without loss of generality, we can assume that  $N_0$ is an essential submodule of $N.$ In this case $\Soc(N)=\Soc(N_0).$

Consider the following three cases.

{\it  Case }  $f(J(N)\cap \Soc(N))=0, f(SI(N))=0.$   There exists a homomorphism $g:N\to E(M),$ such that the equality holds $f=g\iota.$ If $g(N)SI(R_R)\neq 0,$ then exists a primitive idempotent $e\in R$ such that $eR$ is a simple injective module and $g(N)e\neq 0.$ Then $neR$ is a  simple injective module and $f(neR)=g(neR)\neq 0$ for some $n\in N,$ which contradicts the equality $f(SI(N))=0.$ Thus  $g(N)SI(R_R)=0.$ Since $R/SI(R_R)$ is a right Artinian ring and by Corollary 2.1, $R/SI(R_R)$ is an almost right $V$-ring, then by \cite[Corollary 3.2]{12},  $R_R/SI(R_R)$ is an Artinian serial ring and $J^2(R_R/SI(R_R))=0.$ Then by \cite[13.67]{16}, $g(N)=N_1\oplus N_2,$ where $N_1$ is a semisimple module and $N_2$ is a direct sum of local  modules of length two. If $N_2\neq 0$ then there exists an epimorphism $h:N_2\to L,$ where $L$ is a local module of length two.  Since $L$ is a projective module, $h\pi g$ is a split epimorphism, where $\pi: N_1\oplus N_2\to N_2$ is the natural projection. Consequently, $h\pi g_{\mid L'}$ is an isomorphism for some local submodule $L'$ of the module $N$ and  $f(\Soc(L'))=g(\Soc(L'))\neq 0,$ 
which contradicts the equality $f(J(N)\cap \Soc(N))=0.$ Then $g(N)\subset \Soc(E(M))\subset M.$ Hence, we can conside the homomorphism $g$ as an element of the Abelian group $\Hom_R(N, M).$

{\it  Case } $f(J(N)\cap \Soc(N))\neq 0.$  If $f(N_0)$ is not a $V$-module, then by Lemma 2.1, there exists an epimorphism $h:f(N_0)\to L,$ where $L$ is an uniform but is not a simple module, whose socle is a simple module. By lemma 3.1, $L$ is a projective and injective module. Since $L$ is a projective module, $N_0=f^{-1}(\Ker(h))\oplus L', f(N_0)=\Ker(h)\oplus f(L'),$ where $L'$  is a submodule of $N_0$ and $L\cong L'.$  By Lemma 3.2, the following conditions are satisfied for some direct summands $M',N'$ of modules $M$ and $N$, respectively:
$$
M=M'\oplus f(L'), \Ker(h)\subset M', N=N'\oplus L', f^{-1}(\Ker(h))\subset N'.
$$ 
Let $\pi_1:M'\oplus f(L')\to f(L') ,\pi_2:N'\oplus L'\to L'$ be natural projections. There exists an isomorphism $h':f(L')\to L',$ such that $fh'=1_{f(L')}.$ 
Then we have the equality $(h'\pi_1)f=\pi_2\iota .$

If $f(N_0)$ is a $V$-module, then for some simple submodule $S$ of  $J(N)\cap \Soc(N)$ we have the equality $f(N_0)=f(S)\oplus M',$ where $M'$ is a submodule of the module $M$ and $f(S)\neq 0.$ Let $\pi:f(S)\oplus M'\to f(S)$ be the natural projection. We can consider the homomorphism $f$ as an element of the Abelian group $\Hom_R(N_0, f(N_0)).$  Then $N_0=\Ker(\pi f)\oplus S.$
By Lemma 3.1, the following conditions are satisfied for some submodules  $N'$ and $L'$ of the module $N$: 
$$
N=N'\oplus L', \Ker(\pi f)\subset N', \mathrm{lg}(L')=2, \Soc(L')=S.
$$ 
By Corollary 2.1,  $R$ is a right almost $V$ -ring. Then by \cite[теорема 2.9]{13}, there exists a decomposition $M=M_1\oplus M_2$ of module $M,$ such that  $M_1$ is a complement  for $f(S)$ in  $M$ and $M'\subset M_1.$ Easy to see that $ \pi_2(f(S))$ is a simple essential submodule of $M_2,$ where  $\pi_2:M_1\oplus M_2\to M_2$ is the natural projection. Let $h:S\to \pi_2f(S)$ be the isomorphism is defined by  $h(s)=\pi_2f(s)$ for every $s\in S.$ We can consider the homomorphism $h^{-1}$ as an element of the Abelian group $\Hom_R(\pi_2f(S), L').$  If $M_2$ is a simple module, then we have the equality  $(h^{-1}\pi_2)f=\pi'\iota,$ wherе $\pi':N'\oplus L'\to L'$ is the natural projection. If $M_2$ is not a simple module, then since $M_2$ is an injective module, there is an isomorphism $h':M_2\to L'$ such that $h'_{\mid \pi_2f(S)}=h^{-1}.$ Then we have the equality $(h'\pi_2)f=\pi'\iota.$

{\it  Case } $f(SI(N))\neq 0.$   In this case, for some simple injective submodule $S$ of the module $N$ we have $f(S)\neq 0.$ Since $f(S)$ is an injective module, $M=f(S)\oplus M_0,$ where $M_0$ is a submodule of $M.$ Let $\pi:f(S)\oplus M_0\to f(S)$ be the natural projection. Then $N_0=\Ker(\pi f)\oplus S.$  By Lemma 3.2, there exists a direct summand $N'$ of $N$ such that: 
$$
N=N'\oplus S, \Ker(\pi f)\subset N'.
$$ 
Let $\pi':N'\oplus S\to S$ be the natural projection. There is an isomorphism $h:f(S)\to S,$ such that $fh=1_{f(S)}.$ 
Then we have the equality $(h\pi)f=\pi'\iota.$ 

2)$\Rightarrow $4) Suppose that the ring $R$ satisfy the condition 2). According to the condition 2), we have that $R/J(R)_R=A\oplus B,$ where $A$ is an injective module and $B$ is a $V$-module. By \cite[Theorem 3.2]{20}, $A$ has finite Goldie dimension. Since $A$ is a semiartinian module and $J(A)=0,$ it follows that $A$ is a semisimple module. Therefore, by \cite[23.4]{17}, $R/J(R)$ is a $V$-ring. 

By the condition 2), this implies $R_R=A'\oplus B',$ where $A'$ is an injective module and $B'$ is a $V$-module. It is easy to see, according to the condition 2), the injective hull of every simple $R$-module has the length at most 2. Then by \cite[Theorem 3.2]{20}, $A'$ is a finite direct sum of modules of length at most 2. 

Let $M$ be a right injective $R/SI(R)$-module and $\{L_i\}_{i\in I}$ be a maximal independent set of submodules of  $M$ with $lg(L_i)=2$ for all $i.$ By the condition 2),  $\oplus_{i\in I}L_i$ is an injective $R/SI(R)$-module. Consequently, there exists a submodule $N$ of $M$ such that $M=\oplus_{i\in I}L_i\oplus N.$ If $N(\Soc(R)/SI(R))\neq 0,$ then $N$ contains a simple submodule $S,$ such that $S$ is not injective as right $R$-module.  Then the injective hull $E(S)$ of the right $R$-module $S$ has the length two and obviously $E(S)SI(R)=0.$ Consequently, $S$ is not a  injective right $R/SI(R)$-module and there exists a local injective submodule  $L$ of $N$ of length two  such that $S\subset L.$  This contradicts the choice of the set  $\{L_i\}_{i\in I}.$ Consequently, 
$N(\Soc(R)/SI(R))=0$ and since $R/\Soc(R)$ is a  Artinian semisimple ring, we have $N$ is a semisimple module.   Thus, every injective right $R/SI(R)$-module  is a direct sum of injective hulls of simple modules and since \cite[6.6.4]{15}, we have that $R/SI(R_R)$ is a right Artinian ring. 

3)$\Rightarrow $4) Suppose that the ring $R$ satisfy the condition 3). If $R'=R/J(R)$ is not a right $V$-ring, then by Lemma 2.1, there is a right ideal $T$ of the ring $R'$ such that the right $R'$-module $R'/T$ is an uniform but is not a simple module, whose socle is a simple module.  Consequently, by the condition 3) the module $R'/T$ is projective and isomorphic to a submodule of $R'_{R'},$ which is impossible. Hence, $R/J(R)$ is a $V$-ring. 

Let $S$ be a simple right $R$-module and $E(S)\neq S.$ By condition 3), $E(S)$ is a projective module. By \cite[7.2.8, 11.4.1]{15}, $E(S)$ is a local module. Since $\mathrm{Loewy}(R)\leq 2,$ it follows that $E(S)/S$ is a semisimple module. Consequently, $J(E(S))=S$ and $\mathrm{lg}(E(S))=2.$ 

By Zorn's Lemma there is a maximal independent set of submodules $\{L_i\}_{i\in I}$ of $R_R$ such that $L_i$ is a local injective module of length two for all $i\in I.$ Since  $\mathrm{Loewy}(R_R)\leq 2,$ it follows that $I$ is a finite set and $\mid I \mid < lg(R_R/Soc(R_R)).$ Then $R_R=\oplus_{i\in I}L_i\oplus eR,$ where $e^2=e\in R.$ By Lemma 2.1 and the condition 3), $eR$ is a $V$-module. Consequently, $J(R)=\oplus_{i\in I}L_i.$ 

Now assume that $\Soc(eR)$ contains an infinite set of orthogonal primitive idempotents $\{f_i\}_{i=1}^{\infty}$  with $E(f_i R)\neq f_iR$ for all  $i.$ There exists a subset $I'$ of $I,$ such that $Z(L_i)\neq 0$ for all $i\in I'$ and $fR=\oplus_{i\in I\setminus I'}L_i\oplus eR$ is a nonsingular module, where $f^2=f\in R.$ There exists a homomorphism $f:R_R\to E(\oplus_{i=1}^{\infty}f_i R)$ such that $f(r)=r$ for all $r\in \oplus_{i=1}^{\infty}f_i R.$ Since $E(\oplus_{i=1}^{\infty}f_i R)$ is a nonsingular module, it is generated by the module $\oplus_{i\in I\setminus I'}L_i\oplus eR.$ From the condition 3), implies that $E(\oplus_{i=1}^{\infty}f_i R)$ is a projective module. Consequently, $E(\oplus_{i=1}^{\infty}f_i R)$ can be considered as a direct summand of $\oplus_{i\in I''} M_i,$ where  $M_i\cong fR$ for all $i\in I''.$ There exists a finite subset $\{i_1,\ldots, i_k\}$ of $I''$ such that the following inclusion holds $f(R_R)\subset M_{i_1}\oplus\ldots \oplus M_{i_k}$. Let $\pi:M_{i_1}\oplus\ldots \oplus M_{i_k}\oplus (\oplus_{i\in I''\setminus\{i_1,\ldots, i_k\}} M_{i})\rightarrow \oplus_{i\in I''\setminus\{i_1,\ldots, i_k\}} M_{i}$ be the natural projection. Since  $\oplus_{i\in I''} M_i$ is nonsingular and $f(R_R)$ is an essential submodule of $E(\oplus_{i=1}^{\infty}f_i R),$ then $\pi(E(\oplus_{i=1}^{\infty}f_i R))=0$. Then $E(\oplus_{i=1}^{\infty}f_i R)$ is a direct summand of $M_{i_1}\oplus\ldots \oplus M_{i_k}$, and consequently,  $E(\oplus_{i=1}^{\infty}f_i R)$ is finitely generated. By  \cite[Theorem 3.2]{20}, $E(\oplus_{i=1}^{\infty}f_i R)$ has finite Goldie dimension, which is impossible. Thus,  $\Soc(eR)=SI(R)\oplus S,$  where $S$ is a semisimple module of finite length. Consequently, $R/SI(R_R)$ is a right Artinian ring.

4)$\Rightarrow $5) Suppose that the ring $R$ satisfy the condition 4). There is an idempotent $e\in R$ such that $eR=\oplus_{i\in I}e_iR.$ It is clear that $eRSI(R)=0$ and $SI(R)\subset (1-e)R.$ By the condition 4),  $(1-e)R$ is a semiartinian $V$-module, then $(1-e)Re=0.$ Easy to see that $(1-\overline{e})R/J(R)(1-\overline{e})\cong (1-e)R(1-e),$ where $\overline{e}=e+J(R).$ By \cite[Theorem 2.9]{21}, $(1-\overline{e})R/J(R)(1-\overline{e})\cong End_{R/J(R)}(1-\overline{e})R/J(R)$ is a right $SV$-ring and $\mathrm{Loewy}((1-\overline{e})R/J(R)(1-\overline{e}))\leq 2.$ Thus, the Peirce decomposition $\left(\begin{tabular}{c c} $eRe$ & $_{eRe} eR(1-e)_{(1-e)R(1-e)}$\\
0& $(1-e)R(1-e)$\end{tabular}\right)$ of the ring $R$ satisfies the conditions a) and b) of 4). By Lemma 3.1, every right module over the ring $R/SI(R)$ is a direct sum of an injective module and a $V$-module. It is clear that
every $V$-module over a right Artinian ring is semisimple, then, by \cite[13.67]{16}, $R/SI(R)\cong \left(\begin{tabular}{c c} $eRe$ & $_{eRe} eR(1-e)_{(1-e)R(1-e)/SI(R)}$\\
0& $(1-e)R(1-e)/SI(R)$\end{tabular}\right)$ is an Artinian serial ring whose the square of the Jacobson radical is zero.

5)$\Rightarrow $4)  Put 
$$
R'=  \left(\begin{tabular}{c c} $T$ & $_TM_S$\\
0& $S$\end{tabular}\right), I'=\left(\begin{tabular}{c c} $0$ & $0$\\
0& $I$\end{tabular}\right), e=\left(\begin{tabular}{c c} $1$ & $0$\\
0& $0$\end{tabular}\right), f=\left(\begin{tabular}{c c} $0$ & $0$\\
0& $1$\end{tabular}\right).
$$

Since $eR'I'=0$ and $R'/I'$ is an Artinian serial ring whose the square of the Jacobson radical is zero, there exists a finite set of orthogonal idempotents $\{e_i\}_{i\in I}$ and a semisimple submodule $A$ of $R'_{R'}$ such that $eR'_{R'}=\oplus_{i\in I} e_iR'\oplus A$, and for every $i$ $e_i R'$ is a local right $R'$-module  of length two and $e_i R'$   as right $R/I'$-module is injective. We claim that $e_i R'$ is an injective $R'$-module for every $i$. Suppose that $E(e_i R')I'\neq 0.$ Then, there exists an elements $r\in I', m\in E(e_i R')$ such that $mrR'=\Soc(e_iR').$ Since $e_iR'I'=0$ and $S$ is a regular ring, then $Soc(e_iR')=mrR'=mrR'rR'=0$. This is a contradiction.  Thus, $e_i R'$ is an injective module for every $i$. Since $S\cong R'/eR'$ is a right $V$-ring and $fR'eR'=0,$ we have $fR'$ is a $V$-module.
Since 
$$
R'_{R'}=\oplus_{i\in I} e_i R'\oplus A\oplus fR'
$$ and $ A\oplus fR'$ is a $V$-module, we have that $J(R')=\oplus_{i\in I}J(e_iR')$ and $\mathrm{Loewy}(R'_{R'})\leq 2.$
Since $R'/J(R')\cong T/J(T)\times S,$ it follows that $R'/J(R')$ is a right $SV$-ring.

There exists a right ideal $I''$ of $R'$ such that 
$$
\Soc(fR'_{R'})=I''\oplus (\Soc(fR'_{R'})\cap I').
$$ 
Since the right $R'$-module $I''$ isomorphic to the submodule $fR'_{R'}/I'$ and $\mathrm{lg}(fR'_{R'}/I')<\infty$, we have that $\mathrm{lg}(I'')<\infty.$ Let $N$ is a simple submodule of $\Soc(fR'_{R'})\cap I'.$ We show that $N$ is an injective module. Assume that $E(N)e\neq 0.$ Then, there exists  elements $r\in R', n\in E(N)$ such that $nerR'=N.$ Since $eR'I'=0$ and $S$ is a regular ring, we have $NI'=N$, and consequently $N=NI'=nerR'I'=0,$ which is impossible. Thus $NeR=0.$ Consequently, we can consider $N$ as a module over the ring $R'/eR'.$ Since $R'/eR'\cong S$ is a right $V$-ring, it follows that $E(N)=N.$ Thus  $\Soc(fR'_{R'})\cap I'=\Soc(I')\subset SI(R').$  Since $R'/\Soc(R')_{R'}$ is a semisimple module and $I'/\Soc(I')_{R'}$ isomorphic to a submodule of $R'/\Soc(R')_{R'},$ we have that $I'/\Soc(I')_{R'}$ is a module of finite length. Since $R/I'_{R'}, I'/\Soc(I')_{R'}$ are modules of finite length, we have $R'/\Soc(I')_{R'}$ is a module of finite length. Consequently, $R'/SI(R')$ is a right Artinian ring.~\hfill
\end{proof}

\begin{theorem} 
For a ring $R$ the following conditions are equivalent:
\begin{itemize}
\item [1)] Every $R$-module is almost injective.
\item [2)] The ring $R$ is a direct product of the $SV$-ring whith $Loewy(R_R)\leq 2,$
and an artinian serial ring,
with the square of the Jacobson radical equal to zero.
\end{itemize}
\end{theorem} 

\begin{proof}
The implication 2)$\Rightarrow $1) follows from the previous theorem.

1)$\Rightarrow $2) According to Theorem 3.1, the ring $R$ isomorphic to the formal upper triangular matrix ring $R'=\left(\begin{tabular}{c c} $T$ & $M$\\
0& $S$\end{tabular}\right)$, satisfying the conditions of Theorem 3.1. 5).  Since every left $R'$-module is almost injective, from the analogue of Theorem 7 on the left-hand side, it implies that $J(R')$ contained in a finite direct sum of left local injective  $R'$-modules of length two.
Since $M'=\left(\begin{tabular}{c c} $0$ & $M$\\
0& $0$\end{tabular}\right)\subset J(R'),$ it follows that
$M'=J(\sum_{i=1}^{n} R'e'_{i})=\sum_{i=1}^{n}J(R')e'_{i}$, where
$e'_{1},\ldots, e'_{n}$ are orthogonal primitive idempotents and
$R'e'_{i}$ is a local injective module of length two for every $1\leq
i\leq n$. For every $1\leq i\leq n$, the idempotent $e'_{i}$
has the form $\left(\begin{tabular}{c c} $f_{i}$ & $m_{i}$\\
0& $e_{i}$\end{tabular}\right)$, where $f_{i}, e_{i}$ are idempotents respectively rings $T$ and $S$. Since $J(R')\left(\begin{tabular}{c c} $f_{i}$ & $m_{i}$\\
0& $e_{i}$\end{tabular}\right)=\left(\begin{tabular}{c c} $J(T)f_{i}$ & $Me_{i}$\\
0& $0$\end{tabular}\right)$ is a simple submodule of the left $R$-module $M',$ it follows that $Me_{i}\neq 0$, and consequently
$e_{i}\neq 0$. Since $e'_{i}+J(R')$ is a primitive idempotent of the ring
$R'/J(R')$, we have $f_{i}=0$. Thus, 
$e'_{i}=\left(\begin{tabular}{c c} $0$ & $m_i$\\
0& $e_{i}$\end{tabular}\right)$, where $e_{i}$ is a primitive idempotent of the ring $S$ and $m_ie_i=m_i$. Since  $M'=\left(\begin{tabular}{c c} $0$ & $M$\\
0& $0$\end{tabular}\right)=\sum_{i=1}^{n}J(R')e'_{i}$, then $M=\oplus_{i=1}^{n}Me_{i}$ is a decomposition of the semisimple left $T$-module into a direct sum of simple submodules
and $M(1-\sum_{i=1}^{n}e_{i})=0$. If there exists a primitive
idempotent $e$ of the ring $S$ such that $eS\cong e_{i}S$, where
$1\leq i\leq n$, then $Me\neq 0$. Then the right
ideals $(\sum_{i=1}^{n}e_{i})S$ and $(1-\sum_{i=1}^{n}e_{i})S$ of the ring $S$ do not contain isomorphic simple right $R$-submodules.
Consequently, $e=\sum_{i=1}^{n}e_{i}$ is a central idempotent
of the ring $S$ and the ring $R$ is isomorphic to the direct product of the $SV$-ring
$(1-e)S$ and the Artinian serial ring $\left(\begin{tabular}{c c} $T$ & $M$\\
0& $eS$\end{tabular}\right)$ whose the square of the Jacobson radical is zero.
\end{proof}

The following theorem follows from the previous theorem and \cite[theorem 1.7]{22}.

\begin{theorem} For commutative rings $R$ the following conditions are equivalent:
\begin{itemize}
\item [1)] Every $R$-module is almost injective;
\item [2)] Every $R$-module  is an extension of the semisimple module by an injective one. 
\end{itemize}
\end{theorem}


\begin{thebibliography}{99}
\bibitem{1} Y.~Baba, Note on almost M-injectives,    Osaka J. Math.  26(1989)  687–698 

\bibitem{2} M.~Harada, T.~Mabuchi, On almost M-projectives, Osaka J. Math. l 26(1989)  837–848
 

\bibitem{4} Y.~Baba, M.~Harada, On almost M-projectives and almost M-injectives, Tsukuba J. Math. 14(1990)  53–69


\bibitem{5}  M.~Harada, On almost relative injectives on Artinian modules, Osaka J. Math.  27(1990) 963–971


\bibitem{6}  M.~Harada, Direct sums of almost relative injective modules,  Osaka J. Math.  28(1991) 751–758


\bibitem{7}  M.~Harada, Note on almost relative projectives and almost relative injectives, Osaka J. Math. 29(1992)  435–446


\bibitem{9} M.~Harada,  Almost projective modules, J. Algebra  159(1993)
 150–157



\bibitem{10} A.~Alahmadi, S. K. Jain, A note on almost injective modules, Math. J. Okayam, 51(2009)  101-109


\bibitem{11} A.~Alahmadi, S. K. Jain, S.~Singh, Characterizations of Almost Injective
Modules, Contemp. Math. 634(2015) 11-17



\bibitem{12}  M.~Arabi-Kakavand, S.~Asgari, Y.~Tolooe,  Rings Over Which Every Module Is Almost Injective, Communications in Algebra 44(7)(2016)  2908-2918


\bibitem{13}  M.~Arabi-Kakavand, S.~Asgari, H.~Khabazian, Rings for which every simple module is almost injective, Bull. Iranian Math. Soc. 42(1)(2016) 113-127



\bibitem{14} S.~Singh, Almost relative  injective modules, Osaka J. Math. 53( 2016) 425–438


\bibitem{15}  Kasch, F. Modules and Rings, Academic Press 1982.


\bibitem{16}  A.A. Tuganbaev, Ring Theory. Arithmetical Modules and Rings, MCCME, Moscow, 2009, 472 с. 


\bibitem{17} R.~Wisbauer, Foundations of Module and Ring Theory 
 Philadelphia: Gordon and
Breach  1991


\bibitem{18}  A. N. Abyzov, Weakly regular modules over normal rings, Sibirsk. Mat. Zh., 49:4 (2008), 721–738


\bibitem{19}  N.\,V. Dung,  D.\,V. Huynh, P.\,F. Smith, R.~Wisbauer, Extending Modules,  Longman, Harlow Pitman
Research Notes in Mathematics 313 1994



\bibitem{20} H.\,Q. Dinh, D.\,V. Huynh, Some results on self-injective rings and $\Sigma-CS$ rings, Commun. Algebra 31(12)(2003)  6063–6077

 	
\bibitem{21} G.~Baccella, Semi-Artinian V-rings and semi-Artinian von Neumann regular rings,  J. Algebra 173(1995)  587–612



\bibitem{22} A. N. Abyzov, 	Regular semiartinian rings,  Russian Mathematics (Izvestiya VUZ. Matematika), 2012, 56:1, 1–8


\bibitem{23}   	P. A. Krylov, A. A. Tuganbaev, “Modules over formal matrix rings”, Fundament. i prikl. matem., 15:8 (2009), 145–211 

\end{thebibliography}
\end{document}